\documentclass[a4paper,10pt]{article}

\pdfoutput=1

\usepackage[utf8]{inputenc}
\usepackage{amssymb,latexsym,amsmath,epsfig,amsthm}

\newtheorem{theorem}{Theorem}
\newtheorem{corollary}[theorem]{Corollary}

\newtheorem{conjecture}[theorem]{Conjecture}

\newenvironment{definition}[1][Definition]{\begin{trivlist}
\item[\hskip \labelsep {\bfseries #1}]}{\end{trivlist}}

\title{Some Thoughts on Determining Symmetric Palintiples}
\author{Benjamin V. Holt}

\begin{document}

\maketitle

\begin{abstract}
A palintiple is a natural number which is an integer multiple of its digit reversal. A previous paper partitions all palintiples into three distinct classes according to patterns in the carries and then determines all palintiples belonging to the shifted-symmetric class. In this paper we consider a strategy for finding all symmetric palintiples based upon recent work. We also discuss the last case of asymmetric palintiples and consider bases for which asymmetric palintiples do not exist.
\end{abstract}

\section{Introduction}

Natural numbers which are integer multiples of 
their digit reversals are called \textit{palintiples}. 
The most well-known examples of base-10 palintiples include $87912=4\cdot 21978$ and $98901 = 9 \cdot 10989$. 

In general, letting $p=(d_k, d_{k-1},\ldots, d_0)_b$ represent the natural number 
$\sum_{j=0}^{k}d_jb^j$ where $b>1$ is a natural number and $0\leq d_j < b$ for all $0\leq j \leq k$,
we say $p$ is an $(n,b)$-\textit{palintiple} provided 
\[
(d_k, d_{k-1},\ldots, d_0)_b=n(d_0, \ldots, d_{k-1}, d_k)_b.
\]
for some natural number $n$. Base-$b$ palindromes and examples with leading zeros are not considered
so that $1<n<b$ and $b>2$. 

Previous work on palintiple numbers mostly restricts its attention to a limited number of digits 
\cite{sutcliffe, kacz, kolsmol, pudwell} or on characterizing base-10 palintiples \cite{hoey, web_wil}. 
Other work includes constructive methods for finding palintiples 
when they exist \cite{young_1,sloane}. 
Both of these papers use a graph-theoretical framework to gain insight into the problem.
The graphs essentially visualize an efficient search a for palintiples using the potential carries as nodes.
As Sloane \cite{sloane} and Young \cite{young_1,young_2} have noted, 
such graphs have interesting properties in and of themselves.
The work of \cite{holt} establishes some general properties of palintiple 
numbers of any base having an arbitrary number of digits. The methods therein
pay particular attention to the structure of the carries which naturally separates all 
palintiples into three mutually exclusive and exhaustive classes. 
Letting $p=(d_k, d_{k-1},\ldots, d_0)_b$ be an $(n,b)$-palintiple with carries $c_k$,
$c_{k-1}$,$\ldots$, $c_0$, these classes are defined as follows: we say that $p$ is \textit{symmetric} if
$c_j=c_{k-j}$ for all $0\leq j \leq k$ 
and that $p$ is \textit{shifted-symmetric} if $c_j=c_{k-j+1}$ for all $0 \leq j \leq k$.
A palintiple that is neither symmetric nor shifted-symmetric is called \textit {asymmetric}.
The examples given above are symmetric. 

Among other things, \cite{holt} finds all shifted-symmetric palintiples. 
This paper continues these efforts by further conjecturing the nature of 
all symmetric palintiples while further characterizing them.
A characterization of shifted-symmetric palintiples is also presented in order to gain insight into the asymmetric case. 
The result is a characterization of asymmetric palintiples which involves a only a simple condition on the base and multiplier.
Lastly, we address the problem of bases for which asymmetric palintiples do not exist.   

For more examples of palintiples as well as a more thorough introduction to the topic, the reader is directed 
to \cite{holt} and \cite{sloane}. 

\section{Determining All Symmetric Palintiples}

For convenience, we state two results found in \cite{young_1, holt, sloane}. 
Let $(d_k, d_{k-1},\ldots, d_0)_b$ be an $(n,b)$-palintiple and let $c_j$ be the $j$th carry. Then
\begin{equation}
b c_{j+1}-c_j=nd_{k-j}-d_{j}
\label{fund}
\end{equation}
for $0\leq j \leq k$.
From these equations the digits may be stated in terms of the carries:
\begin{equation}
d_j=\frac{nbc_{k-j+1}-nc_{k-j}+bc_{j+1}-c_j}{n^2-1}
\label{digits}
\end{equation}
for $0\leq j \leq k$.

We now consider a conjecture which, if proved, would answer an open question posed by \cite{holt,sloane} and
further characterize symmetric palintiples.

\begin{conjecture}
Suppose $p=(d_k,d_{k-1}, \ldots, d_0)_b$ is an $(n,b)$-palintiple with carries $c_k, c_{k-1},\ldots,c_0$.
Then, if $p$ is symmetric, then $c_j \equiv 0 \mod (n-1)$ for $0 \leq j \leq k$.
\label{carries_mod_n_minus_one}
\end{conjecture}

\begin{corollary}
If the Conjecture \ref{carries_mod_n_minus_one} is true, 
then an $(n,b)$-palintiple is symmetric if and only if $n+1$ divides $b$.
\label{n_plus_one_divides_b}
\end{corollary}
\begin{proof}
Suppose $p=(d_k,d_{k-1}, \ldots, d_0)_b$ is an $(n,b)$-palintiple with
$c_k, c_{k-1},\ldots,c_0$ as its carries.
The reverse implication is established in \cite{holt}. 
Suppose then $p$ is symmetric.
Then $c_k=c_0=0$. Supposing Conjecture \ref{carries_mod_n_minus_one} holds, $c_1 \equiv 0 \mod (n-1)$. 
If $c_1=0$ then $d_0=0$ by Equation \ref{digits}. 
Therefore since $c_j \leq n-1$ for all $0 \leq j \leq k$ as shown in \cite{holt}, $c_1=n-1$. 
Another application of Equation \ref{digits} finishes the proof. 
\end{proof}

Provided Conjecture \ref{carries_mod_n_minus_one}, then the following characterization of symmetric palintiples also holds.

\begin{conjecture}
$p=(d_k,d_{k-1}, \ldots, d_0)_b$ is a symmetric $(n,b)$-palintiple 
if and only if the following conditions hold:
\begin{enumerate}

\item $d_j=nqr_{k-j+1}+qr_{j+1}-r_j$ for all $0 \leq j \leq k$ where $q$ is an integer such that $b=q(n+1)$
and $r_k, r_{k-1},\ldots, r_0$ is a palindromic binary sequence.

\item $r_0=r_k=0$ and $r_1=r_{k-1}=1$. 

\item $r_{j-1}=r_{j+1}=1$ implies $r_j=1$ and $r_{j-1}=r_{j+1}=0$ implies $r_j=0$ 
for all $0<j<k$. 
\end{enumerate}
\label{reg1}
\end{conjecture}

\begin{proof}

Suppose $p$ is a symmetric $(n,b)$-palintiple with carries $c_k, c_{k-1},\ldots,c_0$. 
Corollary \ref{n_plus_one_divides_b} implies that $b=q(n+1)$ for some integer $q$.
Since $c_j \leq n-1$ for all $0 \leq j \leq k$, 
an application of Equation \ref{digits} and the \textbf{Conjecture} 
show that the first condition is satisfied.

By definition $c_0=c_k=0$ so that $r_0=r_k=0$, and if $r_1=r_{k-1}\neq 1$, then $d_0=d_k=0$
which is a violation of the condition that there are no leading zeros.

If $r_{j-1}=r_{j+1}=1$ but $r_j=0$, then $d_j=b$ which is not a base-$b$ digit.
Similarly,  if $r_{j-1}=r_{j+1}=0$ but $r_j=1$, then $d_j=-1$.

Now suppose conditions 1, 2, and 3 hold. Conditions 2 and 3 guarantee that $0 \leq d_j<b$
and that there are no leading zeros. 
From here it is a routine matter to show that $p$ is an $(n,b)$-palintiple. Condition 1
guarantees that $n+1$ divides $b$ which, by Corollary \ref{n_plus_one_divides_b}, proves that $p$ is symmetric.
\end{proof}


Casting the above into the light of recent work by Kendrick \cite{kendrick} who showed that 
the $(n,b)$-Young graph, denoted here as $Y(n,b)$, is isomorphic to $Y(10,9)$ 
(otherwise known as the ``1089 graph'' since $9801=9\cdot 1089$ is a $(9,10)$-palintiple)
if and only if $n+1$ divides $b$ (as conjectured by Sloane \cite{sloane}), 
we ask if the following are equivalent for an $(n,b)$-palintiple $p$
with carries $c_k, c_{k-1}, \cdots, c_1, c_0$: 
\begin{enumerate}
\item $p$ is symmetric 
\item $Y(n,b) \backsimeq Y(10,9)$  
\item $c_j \equiv 0 \mod(n-1)$
\item $n+1$ divides $b$. 
\end{enumerate}

If $Y(n,b) \backsimeq Y(10,9)$, the work of Kendrick \cite{kendrick} 
shows that any node of has the form $[0,0]$, $[0,n-1]$, $[n-1,0]$, or $[n-1,n-1]$ 
which establishes $(2)\Longrightarrow(3)$.
 $(3)\Longrightarrow(4)$ is easily established since by Equation \ref{digits},
$d_0=\frac{bc_1}{n^2-1}$ and $d_0 \neq 0$.
$(4)\Longrightarrow(1)$ is demonstrated in \cite{holt}.

Proving the above equivalence (which as demonstrated above is tantamount to proving  
(1) $\Longrightarrow$ (2), and we leave it as an open question)
would prove Conjecture \ref{carries_mod_n_minus_one}.

\section{Some Notes on Asymmetric Palintiples}

In contrast to their symmetric and shifted-symmetric counterparts, 
it is not well understood under what conditions asymmetric palintiples exist.  
Whereas, conditions on $n$ and $b$ which characterize symmetric and shifted-symmetric palintiples  
given by Corollary \ref{reg1} and Theorem \ref{reg2condition} below also assure their existence
(see \cite{holt} and the proof of Theorem \ref{regbases} below), 
the same cannot be said for asymmetric palintiples.
Moreover, the behavior of asymmetric palintiples is considerably more haphazard.
On one hand, $k+1$-digit symmetric palintiples exist in every base for all $k \geq 3$.
Similarly, the existence of a shifted-symmetric $(n,b)$-palintiple ensures the existence of 
a $k+1$-digit example for all $k \geq 1$.
On the other hand, minimal examples of asymmetric palintiples might exist for only a certain number of digits.
(A $k+1$-digit $(n,b)$-palintiple is considered minimal if no $(n,b)$-palintiples exist for less than $k+1$ digits.)
This minimal number of digits can be quite large: the smallest $(106,420)$-palintiple is 105 digits long. 
Furthermore, unlike the symmetric and shifted-symmetric cases, 
the existence of a $k+1$-digit asymmetric $(n,b)$-palintiple does not guarantee the existence of a $k+2$-digit example.
(We note, however, that symmetrically concatenating digit strings of $(n,b)$-palintiples yields other representations
of larger $(n,b)$-palintiples (\cite{sutcliffe, kolsmol, young_1}). 
Thus, the existence of a $k+1$-digit minimal example always guarantees the existence of a $m(k+1)$-digit example 
for any natural number $m$.)
What is more, asymmetric palintiples seem to completely lack the kinds of patterns which are found in the carries and digits 
(Figure 1) of symmetric and shifted-symmetric palintiples.

\begin{center}
 \makebox[\textwidth][r]{\includegraphics[width=1\textwidth]{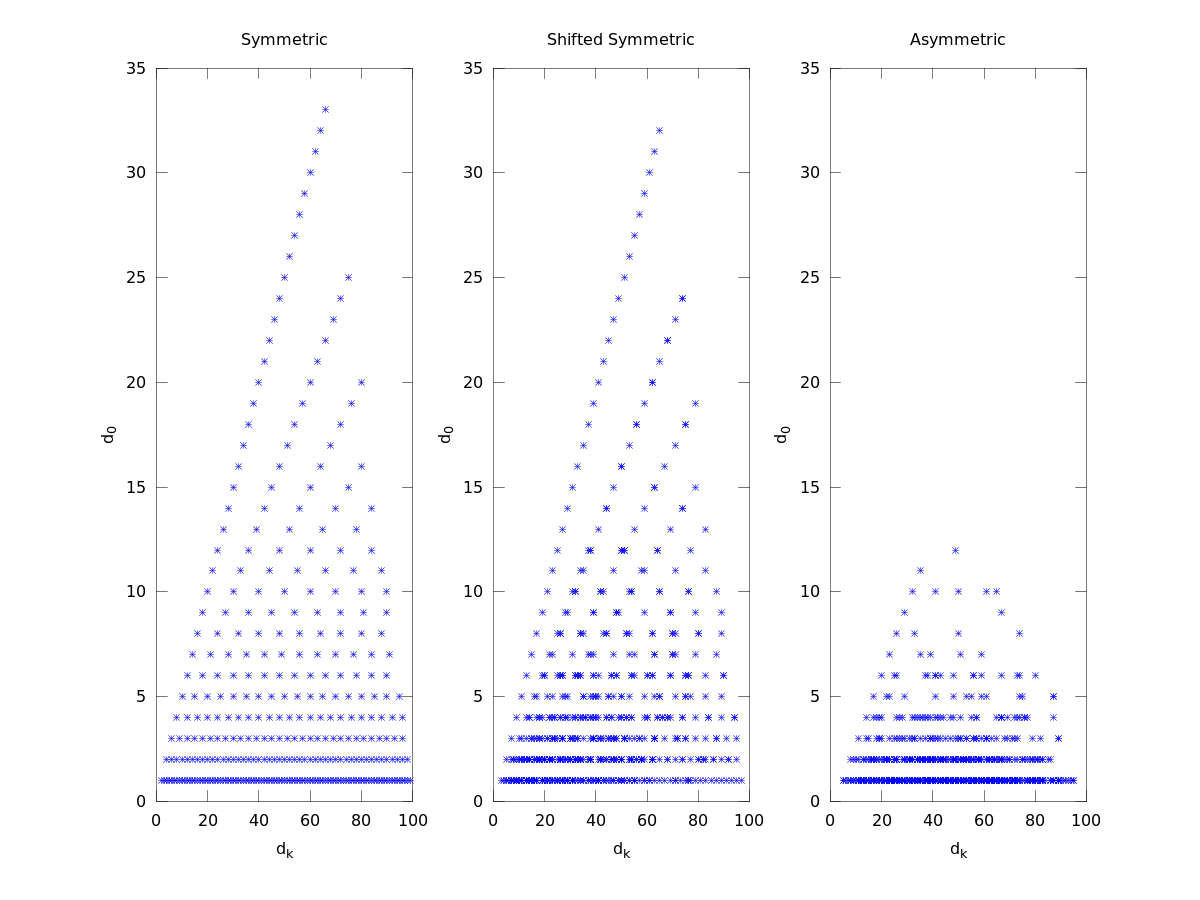}}
 \textbf{Figure 1}: A plot of $d_0$ versus $d_k$ for all minimal palintiples for $3 \leq b \leq 100$.
\end{center}

The next theorem states a simple condition between $n$ and $b$ which characterizes shifted-symmetric palintiples.   

\begin{theorem}
An $(n,b)$-palintiple is shifted-symmetric if and only if $\gcd(b-n,n^2-1)\geq n+1$.
\label{reg2condition}
\end{theorem}
\begin{proof}
Suppose $p=(d_k, d_{k-1},\ldots, d_0)_b$ is an $(n,b)$-palintiple with carries $c_k$, $c_{k-1}$,$\ldots$, $c_0$.
If $p$ is shifted-symmetric, then, by results in \cite{holt}, $(b-n)c_j \equiv 0 \mod (n^2-1)$ must have non-trivial solutions
less than or equal to $n-1$. That is, $\frac{n^2-1}{\gcd(b-n,n^2-1)}\leq n-1$ which establishes the forward
implication.

Suppose then that $\gcd(b-n,n^2-1)\geq n+1$. Equation \ref{digits} implies that $(n^2-1)d_0=bc_1-nc_k$.
Consequently, $(b-n)c_1 \equiv n(c_k-c_1) \mod (n^2-1)$. Hence, $\gcd(b-n,n^2-1)$ divides $c_k-c_1$.
Without loss of generality suppose $c_k \geq c_1$. Then by our hypothesis, it follows that 
$c_k-c_1=\alpha \gcd(b-n,n^2-1) \geq \alpha (n+1)$ for some integer $\alpha \geq 0$.
But since $c_j \leq n-1$ for all $0 \leq j \leq k$, 
we see that $\alpha=0$ so that $c_1=c_k$. 
Equation \ref{digits} implies that $(b-n)c_1 \equiv 0 \mod (n^2-1)$.

Suppose now that $c_j=c_{k-j+1}$ and that $(b-n)c_j \equiv 0 \mod (n^2-1)$. Equation \ref{digits} implies
$(n^2-1)d_j=(nb-1)c_j+bc_{j+1}-nc_{k-j}$. This with our inductive hypotheses proves that
$bc_{j+1} \equiv nc_{k-j} \mod (n^2-1)$, that is, $(b-n)c_{j+1} \equiv n(c_{k-j}-c_{j+1}) \mod (n^2-1)$.
The same argument given above for $j=1$ shows that $c_{j+1}=c_{k-j}$ and the proof is complete.
\end{proof}

\begin{corollary}
If Conjecture \ref{carries_mod_n_minus_one} holds, 
then an $(n,b)$-palintiple is asymmetric if and only if $\gcd(b-n,n^2-1)\leq n-1$ and $(n+1)\nmid b$.
\label{irreg}
\end{corollary}

\begin{corollary}
If Conjecture \ref{carries_mod_n_minus_one} holds, 
then two $(n,b)$-palintiples are both either symmetric, shifted-symmetric, or asymmetric.
\label{pal_type}
\end{corollary}

Corollary \ref{irreg} sheds some light upon the case of asymmetric palintiples by
narrowing down possible values of $n$ and $b$ for which they may exist. 
It should be noted, however, that this result does not specify conditions which guarantee 
their existence. Finding such conditions is key to finding all palintiples.

\section{Symmetric Bases}

The work of Kendrick \cite{kendrick} reveals the asymmetric class to be more diverse than previously thought.
Given the large variety of Young graph isomorphism classes that arise for even relatively low bases ($b \leq 336$)
and the apparent rate at which the number of isomorphism classes grows with $b$,
the difficulties presented by the asymmetric class are not surprising.

Such difficulties naturally lead to the question of bases which do not allow for their 
existence (that is, bases for which all palintiples are known). 
Considering bases up to 20, asymmetric palintiples are known to exist in bases 8, 11, 14, 15, 17, 18, 19, and 20.
On the other hand, methods developed by Young \cite{young_1} along with Corollary \ref{irreg} show that if
Conjecture \ref{carries_mod_n_minus_one} holds, that
no asymmetric palintiples exist in bases 5, 7, 9, and 13. 
In similar fashion, it is conjectured that only symmetric palintiples exist in bases 10, 12, and 16.
Arguments in \cite{holt} show that no asymmetric palintiples exist in bases 3, 4, and 6.

We note that no bases exist for which only shifted-symmetric palintiples exist since symmetric palintiples exist
for every base.  With these facts in mind we make the following definition.

\begin{definition}
 A base for which no asymmetric palintiples exist is called \textit{symmetric}. A base for which only 
 symmetric palintiples exist is called \textit{strongly symmetric}.
\end{definition}

It is no coincidence that, with the exception of 3, 
all known strongly symmetric bases are one less than a prime number as the following shows.

\begin{theorem}
Suppose $b>3$ is a strongly symmetric base. Then $b+1$ is prime. 
 \label{regbases}
\end{theorem}

\begin{proof}
Suppose $b+1$ is composite. Choose $1<n<b$ such that $n+1$ divides $b+1$.
Then $\gcd(b-n,n+1)=\gcd(b+1,n+1)=n+1$. Therefore $\gcd(b-n,n^2-1)\geq n+1$
so that there is at least one non-trivial solution $c$ to the congruence $(b-n)c \equiv 0 \mod (n^2-1)$
that is less than or equal to $n-1$. Methods in \cite{holt} then allow us to construct
a shifted-symmetric $(n,b)$-palintiple. Hence, $b$ cannot be a strongly symmetric base. 
\end{proof}

It is still unknown if infinitely many symmetric or strongly symmetric bases exist.
This number may very well be finite:
electronically implementing the methods of Young \cite{young_1} we see that, 
with the exception of those mentioned above, there are no symmetric bases less than 544321.
Thus, the converse of Theorem \ref{regbases} fails to hold in most cases. 
We leave these as open questions.

\end{document}